\documentclass[11pt, amscd]{amsart}
\usepackage{xypic}
\xyoption{all}
\usepackage{amstext}
\usepackage{amsthm,euscript} 
\usepackage{enumerate}
\usepackage{amsopn}
\usepackage{amsfonts}
\usepackage{amsmath}
\usepackage{latexsym}
\usepackage{amscd}
\usepackage{amssymb}
\usepackage{amsmath}
\usepackage{verbatim}
\usepackage{fancyhdr}

\usepackage{parskip}
\usepackage{setspace}

\usepackage[colorlinks=true, pdfstartview=FitV, linkcolor=blue,
citecolor=blue,
urlcolor=blue,breaklinks=true]{hyperref}

\newcommand{\lv}[1]{} 

\usepackage{fancyhdr}

\swapnumbers
\newtheorem{theorem}[subsection]{Theorem}

\newtheorem{lemma}[subsection]{Lemma}
\newtheorem{corollary}[subsection]{Corollary}

\theoremstyle{definition}

\newtheorem{remark}[subsection]{Remark}

\numberwithin{equation}{section} \allowdisplaybreaks

\numberwithin{equation}{subsection}

\def\gm{\mathfrak m}

\def\G{\mathbf G}

\newcommand{\Aut}{\operatorname{Aut}}
 \DeclareMathOperator{\ad}{ad}

\newcommand{\Hom}{\operatorname{Hom}}

\newcommand{\End}{\operatorname{End}}

\newcommand{\Ctd}{\operatorname{Ctd}}

\newcommand{\GL}{{\operatorname{GL}}}

\newcommand{\PGL}{{\operatorname{PGL}}}

\newcommand{\g}{\mathfrak g}

\def\Z{\mathbb Z}

\def\al{\alpha} \def\be{\beta}

\def\si{\sigma}

\newcommand\pa{\partial}

\newcommand\la{\lambda}

\newcommand{\euA}{\EuScript{A}}
\newcommand{\euC}{\EuScript{C}}
\newcommand{\euD}{\EuScript{D}}
\newcommand{\euE}{\EuScript{E}}
\newcommand{\euH}{\EuScript{H}}
\newcommand{\euL}{\EuScript{L}}

\newcommand{\lsl}{\ensuremath{\mathfrak{sl}}}
\newcommand{\gl}{\ensuremath{\mathfrak{gl}}}

\begin{document}

\title[]{ Counterexample to conjugacy of maximal abelian diagonalizable subalgebras
in extended affine Lie algebras}

\author{U. Yahorau}\footnote{The work of the author was supported by the NSERC grant 385795-2010 of Kirill Zainoulline, NSERC grant of Erhard Neher and NSERC grant of Yuly Billig.} 
\address{Department of Mathematics, University of Ottawa, Ottawa, Ontario, K1N 6N5 , Canada}\email{uyahorau@uottawa.ca}

\begin{abstract}
We give a counterexample to the conjugacy of
maximal abelian diagonalizable subalgebras in extended affine Lie algebras.

{\it{Keywords:}} Extended Affine Lie Algebra, Conjugacy, Laurent polynomials.

MSC 2000 17B67, 14L30.
\end{abstract}

\maketitle

\section*{introduction}
That Cartan subalgebras of finite-dimensional simple Lie algebra $\g$ over an algebraically closed field $k$ of characteristic 0 are conjugate is a Chevalley's theorem. They are abelian, act $k$-diagonalizably on $\g$ and are maximal with respect to these properties. It appears that this is the right way to define ``Cartan subalgebras" in the infinite dimensional setup. Celebrated theorem of Peterson and Kac \cite{PK} states that all MADs (maximal abelian diagonalizable subalgebras) of symmetrizable Kac-Moody Lie algebras are conjugate. In the language of extended affine Lie algebras (EALAs for short) finite-dimensional simple Lie algebras are the case of nullity 0, while the affine Kac-Moody Lie algebras are the case of nullity 1. Therefore EALAs can be thought of as the generalization of these classical Lie algebras to higher nullities. In a recent work \cite{CNPY} the conjugacy of all structure MADs of fgc EALAs was proved. However, a natural question was left open: are all maximal abelian $k$-ad-diagonalizable subalgebras of an EALA conjugate?    

The main result of this paper is Theorem~\ref{counterexample}, which answers this question in negative providing a counterexample to a conjugacy of arbitrary MADs in an EALA. This counterexample is ``minimal" in a sense that it is constructed in an EALA of nullity 2, and as we mentioned above the conjugacy always holds in nullities 0 and 1.

The paper is organized as follows. In section~\ref{example} we construct a class of extended affine Lie algebras (parametrized by certain derivations). Then in section~\ref{2} we construct two maximal abelian $k$-ad-diagonalizable subalgebras of this algebra which are not conjugate.

Let us fix notation. Throughout $k$ is an algebraically closed field of characteristic 0 and $R=k[t_1^{\pm{1}},t_2^{\pm{1}}]$ is a ring of Laurent polynomials in 2 variables with coefficients in $k$.

{\bf{Acknowledgments.}} The author would like to thank Vladimir Chernousov, Erhard Neher and Arturo Pianzola for useful discussions.

\section{Example of an Extended Affine Lie Algebra}\label{example}\label{example}
In this section we give an example of an EALA. We will use it in the next section to construct a counterexample to a conjugacy. 

Let $Q=(t_1,t_2)$ be a quaternion Azumaya algebra over $R$.
Thus, it is generated by elements $\imath,\jmath$ subject to relations $\imath^2=t_1, \jmath^2=t_2$ and
$\imath\jmath=-\jmath\imath$. 


One easily checks that $Q$  has a structure of an associative
torus (\cite[Definition 4.20]{Ne1}) of type $\mathbb{Z}^2$. Indeed,
$$
Q=\bigoplus_{(a,b)\in \mathbb{Z}^2}Q^{(a,b)}
$$
is a $\mathbb{Z}^2$-graded Lie algebra over $k$
with one-dimensional graded components
$Q^{(a,b)}=k\imath^a\jmath^b,$ $(a,b)\in \mathbb{Z}^2$, containing invertible elements $\imath^a\jmath^b$.

In general, if $A$ is an associative torus of type $M$ then
$\lsl_n(A)=[\gl_n (A), \gl_n(A)]$ is a Lie torus of type $(A_{n-1},M)$ (see \cite[Definition 4.2]{Ne1} for the definition of a Lie torus and \cite[Exercise 4.21]{Ne1}). Thus, $\euL=\lsl_2(Q)$ is a Lie torus of type $(A_1,\mathbb{Z}^2)$. In particular, it has a double grading such that
$$
\euL^{(a,b)}_0=k \begin{bmatrix}\imath^a\jmath^b   & 0 \\  0&-\imath^a\jmath^b \end{bmatrix}
$$
if $a,b$ are even and
$$
\euL^{(a,b)}_0=k\begin{bmatrix}\imath^a\jmath^b   & 0 \\ 0& -\imath^a\jmath^b  \end{bmatrix}
\oplus
k\begin{bmatrix}\imath^a\jmath^b   & 0 \\  0&\imath^a\jmath^b \end{bmatrix}
$$
otherwise; also we have
$$
\euL^{(a,b)}_{-2}=k \begin{bmatrix} 0 & 0 \\  \imath^a\jmath^b & 0 \end{bmatrix},
$$
$$
\euL^{(a,b)}_2=
k \begin{bmatrix} 0 & \imath^a\jmath^b \\  0 & 0 \end{bmatrix},
$$
where $(a,b)\in \mathbb{Z}^2.$

Define a bilinear form $(-,-)$ on $\euL$ by
$$
\left(
\begin{bmatrix}x_{11} & x_{12}\\ x_{21} & x_{22} \end{bmatrix},
\begin{bmatrix} y_{11} & y_{12}\\ y_{21} & y_{22} \end{bmatrix}
\right)
=(x_{11}y_{11}+x_{12}y_{21}+x_{21}y_{12}+x_{22}y_{22})_0,
$$
where $a_0$ for $a\in Q$ denotes the $Q^{(0,0)}$-component of $a$. This form is
symmetric, nondegenerate, invariant, graded, i.e.
$(\euL^{\lambda}_{\xi},\euL^{\mu}_{\tau})=0$ if $\lambda+\mu\neq 0$ or $\xi + \tau \neq 0$.
(Recall, that a form $(-,-)$ is
{\it{invariant}} if $([a,b],c)=(a,[b,c])$ for all $a,b,c$.)

Any $\theta \in \Hom_\Z(\Z^2, k)$ induces a so-called {\em degree derivation}
$\pa_\theta$ of $\euL$ defined by $\pa_\theta (l^\la) = \theta(\la) l^\la$ for
$l^\la \in \euL^\la$. Take $\theta\in Hom_{\Z}(\Z^2,k)$ such that $\theta((1,0))=1$ and $\theta((0,1))$ is not rational. We put $\euD =k \pa_{\theta}$ and $\euC=\euD^*.$

We set $\euE=\euL\oplus \euC\oplus \euD$
and equip it with the multiplication given by
$$
\begin{array}{ccc}
[l_1\oplus c_1\oplus \pa_1,l_2\oplus c_2\oplus \pa_2]_{\euE}=([l_1,l_2]_{\euL}+\pa_1(l_2)-\pa_2(l_1))\oplus\sigma(l_1,l_2), 
\end{array}
$$
where $\sigma:\euL\times \euL\rightarrow \euC$ is a {\it{central 2-cocycle}} defined by $\sigma(l_1,l_2)(\pa)=(\pa(l_1)|l_2).$

Let  $\euH=k\begin{bmatrix} 1 & 0 \\ 0 & -1 \end{bmatrix}\oplus \euC\oplus \euD$.

The symmetric bilinear form $(-,-)$ on $\euE$ given by
$$
(l_1\oplus c_1\oplus \pa_1,l_2\oplus c_2\oplus \pa_2)=(l_1,l_2)+c_1(\pa_2)+c_2(\pa_1)
$$
is nondegenerate and invariant.

It follows easily from \cite[Theorem 6]{Ne4}, that $(\euE,\euH)$ is an extended affine Lie algebra. In particular, the subalgebra $\euH$ is a MAD of $\euE$ which we call {\it{``standard"}}.

\section{Construction of the Counterexample}\label{2}
In this section we construct a counterexample to conjugacy of MADs in the extended affine Lie algebra $(\euE,\euH)$ of section~\ref{example}. Namely, we construct the MAD $\euH'$ of $\euE$ which is not conjugate to the standard MAD $\euH$.

Let $Q=(t_1,t_2)$ be as above
and let $\euA=M_2(Q)$. Let $V=Q\oplus Q$ be a free right
$Q$-module of rank $2$.
We may view $\euA$ as the algebra $\End_{Q}(V)$ of $Q$-endomorphisms of $V$.
Let $m:V=Q\oplus Q \rightarrow Q$ be a $Q$-linear map given by
$$
(u,v)\mapsto (1+\imath)u-(1+\jmath)v.
$$
Denote its kernel by $W$. It was shown in \cite{GP} that $m$ is split and that
$W$ is a projective $Q$-module of rank 1 which is not free. Since $m$ is split
there is a decomposition $V=W\oplus U$ where $U$ is a free $Q$-module of rank 1.

Let $s\in \End_{Q}(V)$ be the $Q$-linear endomorphism of $V$  which maps $w$ to $-w$ and $u$ to $u$ for all $w\in W$, $u\in U$.

Let $i_W: W\hookrightarrow W\oplus U$, $i_U: U\hookrightarrow W\oplus U$ be canonical inclusions and
$p_W: W\oplus U\twoheadrightarrow W$, $p_U: W\oplus U\twoheadrightarrow U$ be canonical projections.
\begin{lemma}\label{diag}
${\rm ad}(s):\End_{Q}(V)\rightarrow \End_{Q}(V)$ is a $k$-diagonalizable operator.
\end{lemma}
\begin{proof}
Notice that there is a canonical isomorphism of Q-modules
$$
\tau:\End_Q(V)\simeq \End_Q(W)\oplus \Hom_Q(W,U)\oplus \Hom_Q(U,W)\oplus \End_Q(U),
$$
$$
\tau(\phi)=p_W\circ \phi\circ i_W+p_U\circ \phi\circ i_W+ p_W\circ \phi\circ i_U+ p_U\circ \phi \circ i_U\;\; {\rm{for}} \;\; \phi\in \End_{Q}(V).
$$
Let $\phi\in \tau^{-1}({\rm End}_Q(W))$. Then
$$
[\phi,s]=\phi\circ s-s\circ \phi=-\phi-(-\phi)=0.
$$
Similarly, if $\phi\in \tau^{-1}({\rm End}_Q(U))$ then $[\phi,s]=0$. It follows
$$
\tau^{-1}(\End_Q(W)\oplus \End_Q(U))\subset \End_Q(V)_{0}.
$$

Let now $\phi\in \tau^{-1}(\Hom_Q(W,U))$. Then
$$
[\phi,s]=\phi\circ s- s\circ \phi=-\phi-\phi=-2\phi
$$
implying
$$
\tau^{-1}(\Hom_Q(W,U))\subset  \End_Q(V)_{-2}.
$$
Similarly,
$$
\tau^{-1}(\Hom_Q(U,W))\subset  \End_Q(V)_{2}
$$
and the assertion follows.
\end{proof}

Let $S\in \euA$ be the matrix (in the standard basis) of the $Q$-linear endomorpism $s$ of $V$.

\begin{lemma}\label{in sl2}
$S\in \lsl_2(Q).$
\end{lemma}
\begin{proof}
Let $K$ be a field of fractions of $R$. Then $Q_K=Q\otimes_R K$ is a division algebra and $s$ extends to an operator on $V_K=Q_K\oplus Q_K$ by $Q_K$-linearity. It follows from the definition of $s$ that in an appropriate basis of $V_K$ the matrix of $s$ is ${\rm diag}(1,-1).$ The assertion follows.
\end{proof}

\begin{corollary}
$S$ is a $k$-ad-diagonalizable element of $\euL$ whose eigenvalues are $0,\pm 2$.
\end{corollary}
\begin{proof}
Follows from Lemma~\ref{diag} and Lemma~\ref{in sl2}.
\end{proof}

Let $(\euE,\euH)$ be the EALA from Section~\ref{example}.
Our next goal is to show that $S$ is
$k$-ad-diagonalizable considered as an element of $\euE$. We prove this in several steps.
First we show that $S$ is $k$-ad-diagonalizable as an element of $\euE_c:=\euL\oplus \euC.$
\begin{lemma}\label{lemma1}
Let $T\in \euL$ be a $k$-ad-diagonalizable element in $\euL$, i.e.
$$
\euL=\bigoplus\limits_{\alpha} \euL_{\alpha},
$$
where as usual
$$
\euL_{\alpha}=\{x\in \euL\,|\, [T,x]_\euL=\alpha x\}.
$$
Then it is also
$k$-ad-diagonalizable viewed
as an element of $\euE_c$.
\end{lemma}

\begin{proof}
Since $\si$ is a central 2-cocycle,  for $l_\al \in \euL_\al$ and $l_\be \in \euL_\be$ we have
\begin{align*}
  0&= \si(T, [l_\al, l_\be]_{\euL}) + \si(l_\al, [l_\be, T]_{\euL}) + \si(l_\be, [T,l_\al]_{\euL}) \\
     &= \si(T, [l_\al, l_\be]_{\euL}) - \be \si(l_\al, l_\be) + \al \si(l_\be, l_\al).
\end{align*}
Thus, $\si(T,[l_\al, l_\be]_{\euL}) = (\al + \be) \si(l_\al, l_\be)$.
Hence
\begin{align*}
  \big[ T, [l_\al, l_\be]_{\euE_c}\big]_{\euE_c} &=
       \big[ T, [l_\al, l_\be]_{\euL} + \si(l_\al, l_\be) \big]_{\euE_c}
       =   \big[ T, [l_\al,l_\be]_{\euL}\big]_{\euE_c}
  \\& =  \big[ T, [l_\al, l_\be]_{\euL}\big]_{\euL} + \si(T, [l_\al, l_\be]_{\euL})
     = (\al + \be) [l_\al, l_\be]_{\euE_c}
\end{align*}
proving $[\euL_\al, \euL_\be]_{\euE_c} \subset (\euE_c)_{\al + \be}$. It then follows from 
$$
\euE_c=[\euE_c, \euE_c]_{\euE_c}  = [\euL,\euL]_{\euE_c} = \sum_{\al,\be\in k} [\euL_\al,\euL_\be]_{\euE_c} \subset \sum_{\gamma\in k} (\euE_c)_\gamma
$$ 
that $\euE_c$ is spanned by eigenvectors of $\ad T$, whence the result.
\end{proof}

 Let $[S,\pa_{\theta}]_{\euE}=y=y_0+y_2+y_{-2}$, where $y_{\lambda}\in \euL_{\lambda}$.
\begin{lemma}
One has $y_0=0.$
\end{lemma}
\begin{proof}
Assume $y_0\neq 0$. Since $(-,-)|_{\euL_0}$  is nondegenerate there is
$u\in \euL_0$ such that $(u,y_0)\neq 0$. Then taking into consideration
that $(\euL_0,\euL_2)=(\euL_0,\euL_{-2})=0$ we get
$$
0\neq (y_0,u)=
(y_0+y_2+y_{-2},u)=
([S,\pa_{\theta}]_{\euE},u)=
- (\pa_{\theta},[S,u]_{\euE}).
$$
But it follows from Lemma \ref{lemma1}, that $[S,u]_{\euE}=0,$ --
a contradiction.
\end{proof}

Let $\pa'=\pa_{\theta}-\frac{1}{2}y_2+\frac{1}{2}y_{-2}$.
\begin{lemma}\label{lemma4}
One has $[S,\pa']_{\euE}=0.$
\end{lemma}
\begin{proof} We first observe that
$$
\begin{array}{lll}
[S,\pa']_{\euE} & = & [S,\pa_{\theta}-\frac{1}{2}y_2+\frac{1}{2}y_{-2}]_{\euE}\\
& = & y-\frac{1}{2}[S,y_2]_{\euE}+\frac{1}{2}[S,y_{-2}]_{\euE}\\
&= & y-y_2-\frac{1}{2}\si(S,y_2)-y_{-2}+\frac{1}{2}\si(S,y_{-2})\\
&=& \frac{1}{2}\si(S,y_{-2}-y_2)\in \euC.
\end{array}
$$

Hence $[S,\pa']_{\euE}$ is orthogonal to $\euL\oplus \euC.$

Also, using the invariance of the form $(-,-)$ we get
$$
([S,\pa']_{\euE},\pa')=(S,[\pa',\pa']_{\euE})=(S,0)=0,
$$
i.e. $[S,\pa']_{\euE}$ is orthogonal to $\pa'.$ We conclude that $([S,\pa']_{\euE},x)=0$ for any $x\in \euE$. Now the assertion follows from the nondegeneracy of the form $(-,-)$.  

\end{proof}
\begin{corollary}
$S$ is a $k$-ad-diagonalizable element of $\euE$.
\end{corollary}
\begin{proof}
This follows from Lemma~\ref{lemma1} and Lemma~\ref{lemma4}.
\end{proof}

Let $P=a\begin{bmatrix} 1 & 0 \\ 0 & -1 \end{bmatrix}$ where $a\in k^{\times}$ is
an arbitrary scalar and
$\euH=k\cdot P\oplus \euC\oplus \euD$ be
the standard MAD in $\euE$. Let $\euH'$
be any MAD of $\euE$ which contains $S$ (such a MAD does exist by \cite[Lemma 9.3]{CGP}).
We will show that $\euH$ and $\euH'$ are not conjugate in $\euE$.
We will need some auxiliary lemmas.

\begin{lemma}\label{not R-conjugate}
The elements $S$ and $P$ of $\lsl_2(Q)$ are not conjugate
by an $R$-linear automorphism of $\lsl_2(Q).$
\end{lemma}
\begin{proof}
Let  $\bar{ }:Q\rightarrow Q$ be an involution given by $x+y\imath +z\jmath + t\imath \jmath \mapsto x-y \imath -z\jmath - t\imath \jmath$, $x,y,z,t\in R.$
Then any $R$-linear automorphism of $\lsl_2(Q)$ is a conjugation with some matrix in
${\GL}_2(Q)$ or a map
$$
\pi:\begin{bmatrix} a & b \\ c & d \end{bmatrix}\rightarrow
\begin{bmatrix} \overline{a} & \overline{c} \\ \overline{b} & \overline{d} \end{bmatrix}^{-1}
$$
(which is a nontrivial outer $R$-automorphism of $\lsl_2(Q)$)
followed by a conjugation. Assume that
$\phi(P)=S$
for some $\phi\in \Aut_{R-Lie}(\lsl_2(Q)).$

Case $1$: $\phi$ is a conjugation. Then the eigenspaces in $V=Q\oplus Q$ of the
$Q$-linear transformation
$\phi(P)$ are free $Q$-modules of rank 1
(because they are images of those of $P$).
Since $W$ is an eigenspace of $S$ which is not free $Q$-module we get a contradiction.

Case $2$: $\phi$ is $\pi$ followed by a conjugation. But $\pi(P) =P$,
hence we are reduced to the previous case.
 \end{proof}

 Recall that for  an arbitrary $k$-algebra $A$, 
$$
\Ctd_k(A) = \{ \chi \in \End_k(A) : \chi(ab) = \chi(a)b = a\chi(b) \, \forall \, a,b \in A \}.
$$ 
By \cite{BN,GP} we may identify $R \simeq \Ctd_k(\lsl_2(Q))$, $r\mapsto (\chi: x\mapsto rx)$.

The proof of the following lemma is inspired by \cite{CGP2}.
\begin{lemma}\label{not conjugate}
$S$ and $P$ are not conjugate by a $k$-automorphism of $\lsl_2(Q).$
\end{lemma}
\begin{proof}
Assume the contrary. Let $\phi\in \Aut_{k-Lie}(\lsl_2(Q))$ be such that
$\phi(P)=S.$ It induces an automorphism $C(\phi)$ of $\Ctd_k(\lsl_2(Q))$ defined by $\chi\mapsto \phi^{-1}\circ \chi \circ \phi$ for all $\chi\in \Ctd_k(\lsl_2(Q))$.
Consider a new Lie algebra $\euL'=\lsl_2(Q)\otimes_{C(\phi)} R$ over $R$.
As a set it coincides with $\lsl_2(Q)$. Also, the Lie bracket in $\euL'$
is the same as in $\lsl_2(Q)$, but the action of $R$
on $\euL'$ is
given by the composition of $C(\phi)$ and the standard action of $R$ on $\lsl_2(Q)$.
Thus we have a natural $k$-linear Lie algebra isomorphism
$$
\psi: \euL'=\lsl_2(Q)\otimes_{C(\phi)}R\rightarrow \lsl_2(Q)
$$
which sends $P$ to $P$.
It follows from the construction that $\phi \circ \psi:\euL'\to \lsl_2(Q)$ is an
$R$-linear isomorphism.

Note that since the action of $R$ on $\lsl_2(Q)$ is componentwise we have a
natural identification
$$
\euL'=\lsl_2(Q)\otimes_{C(\phi)}R\simeq \lsl_2(Q\otimes_{C(\phi)}R)
$$
and it easily follows from the construction that $Q\otimes_{C(\phi)}R$ is a quaternion algebra
$(\phi(t_1),\phi(t_2))$ over $R$.
Thus $\lsl_2(Q)$ and $\lsl_2(Q\otimes_{C(\phi)} R)$ are $R$-isomorphic and they are $R$-forms of
$\lsl_4(R).$ Moreover they are inner forms, hence correspond to an element
$[\xi]\in H^1(R, \PGL_4)$.

The boundary map $H^1(R,{\rm PGL}_4)\to H^2(R,\G_m)$ maps $[\xi]$
to the Brauer equivalence class of both $Q$ and $Q\otimes_{C(\phi)} R.$
Since $[Q]=[Q\otimes_{C(\phi)} R]$ it follows that there is an $R$-algebra isomorphism
$\overline{\theta}: Q\otimes_{C(\phi)}R \rightarrow Q$ which in turn induces a
canonical $R$-Lie algebra isomorphism
$$
\theta:\lsl_2(Q\otimes_{C(\phi)}R)\rightarrow \lsl_2(Q)
$$
by componentwise application of $\overline{\theta}.$
Clearly, $\theta(P)=P$.

Finally, consider an $R$-linear automorphism
$$
\phi'=\phi\circ\psi\circ\theta^{-1}\in \Aut_{R-Lie}(\lsl_2(Q)).
$$
We have
$$
\phi'(P)=\phi(\psi(\theta^{-1}(P)))=\phi(\psi(P))=\phi(P)=S
$$
which contradicts Lemma~\ref{not R-conjugate}.
\end{proof}

\begin{theorem}\label{counterexample}
There is no $\phi\in \Aut_{k-Lie}(\euE)$ such that $\phi(\euH')=\euH$.
\end{theorem}
\begin{proof}
Assume the contrary. Let $\phi\in \Aut_k(\euE)$ be such that $\phi(\euH')=\euH$. Since $\euE_c=[\euE,\euE]_{\euE}$ we have that $\euE_c$ is $\phi$-stable. Hence
\begin{equation}\label{core}
\phi(\euH'\cap \euE_c)= \euH\cap \euE_c=k\cdot P\oplus \euC.
\end{equation}
Of course, $k\cdot S\oplus \euC\subset \euH'\cap \euE_c$, because $\euC$ is the center of $\euE$.
Therefore by dimension reasons we have $k\cdot S\oplus \euC=\euH'\cap \euE_c$.

The automorphism $\phi$ induces a $k$-automorphism
$\phi_{\euL}:\euL\to \euL$. 
By~(\ref{core}),
$\phi_{\euL}(k\cdot S)=k\cdot P$.
Hence there exists a scalar $\alpha\in k^{\times}$ such that $\phi_{\euL}(S)=\alpha\cdot P$.
 But this contradicts Lemma~\ref{not conjugate}.
\end{proof}

\begin{remark}
That $\euH'$ is not conjugate to $\euH$ by a $k$-automorphism of $\euE$ means that the MAD $\euH'$ is not a structure MAD of $\euE$, i.e. the pair $(\euE, \euH')$ can not be given a structure of an extended affine Lie algebra. However, we do not know which of the axioms of an EALA fails to hold. Moreover, if we had known which axiom fails we would have got the counterexample right away (see the proof of \cite[Proposition 3.4]{CNPY}).  
\end{remark}


\begin{thebibliography}{ABCDE}

\bibitem[BN]{BN} G.~Benkart and E.~Neher, \textit{The centroid of
    extended
    affine and root graded {L}ie algebras}, J. Pure Appl. Algebra
    \textbf{205} (2006), 117--145.


\bibitem[CGP1]{CGP} V.~Chernousov, P.~Gille and A.~Pianzola, {\em Conjugacy theorems for loop reductive group schemes and Lie algebras}, Bull. Math. Sci. (2014) {\bf{4}}, 281-324.

\bibitem[CGP2]{CGP2} V.~Chernousov, P.~Gille and A.~Pianzola, {\em Generalized Onsager algebras and Grothendieck's dessins d'enfants}, arXiv:1311.7097.

\bibitem[CNPY]{CNPY} V.~Chernousov, E.~Neher, A.~Pianzola and U.~Yahorau, {\em On conjugacy of Cartan subalgebras in extended affine Lie algebras}, in preparation.

\bibitem[GP]{GP} P.~Gille and A.~Pianzola, {\em Galois cohomology and forms of algebras over Laurent polynomial rings}, Math. Annalen {\bf{338}} (2007) 497-543.

\bibitem[Ne1]{Ne1} E.~Neher, {\em Extended affine Lie algebras -- an introduction to their structure theory}, Fields Inst. Comm. (2011), 107-167.

\bibitem[Ne2]{Ne4} E.~Neher, {\em Extended affine Lie algebras}, C. R. Math. Acad. Sci. Soc. R. Can. 26 (2004), 90-96.

\bibitem[PK]{PK} D.H.~Peterson and V.~Kac, {\em Infinite flag varieties and conjugacy theorems}, Proc. Natl. Acad. Sci. USA, 80 (1983), 1778-1782.

\end{thebibliography}
\end{document}